\documentclass[11pt,leqno]{article}

\usepackage{amssymb,amsmath,epsfig,amsthm}

\newcommand{\n}{\noindent}

\newcommand{\vp}{\varepsilon}

\numberwithin{equation}{section}

\theoremstyle{plain}
\newtheorem{thm}{Theorem}[section]
\newtheorem{lem}[thm]{Lemma}

\theoremstyle{definition}

\begin{document}
 
\thispagestyle{empty}

\title{Brownian Motion with a Singular Drift}

\author{
Dante DeBlassie\\
Adina Oprisan\\
Robert G. Smits\\
Department of Mathematical Sciences\\
New Mexico State University\\
P. O. Box 30001 \\
Department 3MB\\
Las Cruces, NM \ 88003-8001\\
deblass@nmsu.edu\\
aoprisan@nmsu.edu\\
rsmits@nmsu.edu}

\date{}
\maketitle

\begin{quote}
We study the effect of a power law drift on Brownian motion in the positive half-line, where the order of the drift at $0$ and $\infty$ is different. 
\end{quote}

\bigskip 

\n \emph{2000 Mathematical Subject Classification}. Primary 60J60, 60J65, 60J55. Secondary 60F10.
\bigskip 

\n \emph{Key Words and Phrases}. Brownian motion, Brownian motion with drift, half-line, exit time, stochastic comparison, large deviations, $h$-transform.

\bigskip
\n \emph{Running Title: }Brownian motion with a singular drift. 
\newpage

\section{Introduction}\label{sec1}

\indent For a real-valued process $X_t$, let $\tau_R(X)$ be the first time the process hits $R$. Specialize $X_t$ to be a Brownian motion with a power law drift: the process satisfying the stochastic differential equation
\[
\left\{\begin{array}{l}
dX_t=dB_t-\beta X_t^{-p}\,dt,\quad t<\tau_0(X)\\
\noalign{\smallskip}
X_0=x>0,
\end{array}\right.
\]
where $\beta\not=0$ and $p>0$. 

\bigskip\n Note that when $\beta=0$, $X_t$ is merely Brownian motion and it is well-known (Feller 1971) that
\[
P_x(\tau_0(X)>t)=\frac{2}{\sqrt{2\pi}}\int_0^{x/\sqrt t}e^{-u^2/2}\,dt.
\]

\bigskip\n When $p=1$, it is known that for $\beta>-\frac12$, 
\[
P_x(\tau_0(X)>t)=\frac{2^{1/2-\beta}}{\Gamma(1/2+\beta)}\int_0^{x/\sqrt t}u^{2\beta}e^{-u^2/2}\,dt,
\]
and for $\beta\leq-\frac12$,
\[
P_x(\tau_0(X)=\infty)=1
\]
(G\"oing-Jaeschke and Yor (2003)). Using Theorem 1.1 of Chapter 5 in Pinsky (1995), it is not hard to show if $\beta<0$, then
\[
\left\{\begin{array}{ll}
P_x(\tau_0(X)=\infty)=1,&\quad\text{for $p>1$}\\
&\\
0< P_x(\tau_0(X)=\infty)<1,&\quad\text{for $p<1$.}
\end{array}\right.
\]

\bigskip\n DeBlassie and Smits (2009) proved the following results for $\beta>0$.
\begin{itemize}
\item For $p>1$,

\[\left\{\begin{array}{ll}
E_x[\tau_0(X)^q]<\infty,&\quad\text{if $q<1/2$,}\\
&\\
E_x[\tau_0(X)^q]=\infty,&\quad\text{if $q>1/2$}.
\end{array}\right.
\]

\item For $p<1$,
\[
\lim_{t\to\infty}t^{-(1-p)/(1+p)}P_x(\tau_0(X)>t)=-\gamma(p,\beta),
\]
where
\[
\gamma(p,\beta)=\tfrac12p^{-2p/(1+p)}\beta^{2/(1+p)}\left[B\left(\tfrac12,\tfrac{1-p}{2p}\right)+B\left(\tfrac32,\tfrac{1-p}{2p}\right)\right]B\left(\tfrac12,\tfrac{1-p}{2p}\right)^{-(1-p)/(1+p)}.
\]
\end{itemize}

\bigskip\n In particular, there is a phase transition of sorts as the power $p$ passes through the value $1$: roughly power law behavior (with the same power) of the tail versus subexponential behavior (with a different power).

\bigskip\n For the last case of $\beta>0$ and $0<p<1$, it is not hard to show (we indicate the argument in the next section) that so long as the form of the drift is $\beta x^{-p}$ for large values of $x$, one can change the drift to be of the form $\alpha x^{-q}$, $0<q<1,$ for small values of $x$, keeping it bounded in between, and still have the process hit $0$ almost surely. This holds even if the multiplier $\alpha$ is negative. A natural question suggested by this is to determine how much effect this change has on the asymptotic behavior of the time to hit $0$. We were surprised to learn that the behavior near the origin does not influence the time to hit $0$, at least on the logarithmic level.

\bigskip\n We now state our results precisely. Consider the diffusion $X_t$ given by
\begin{equation}\label{eq1.1}
\left\{\begin{array}{l}
dX_t=dB_t+b(X_t)\,dt\\
\noalign{\smallskip}
X_0=x>0,
\end{array}\right.
\end{equation}
where $B_t$ is one-dimensional Brownian motion and for some $0<M_1<M_2$, $\alpha\in\mathbb{R}$, $\beta>0$ and $0<p,q<1$,
\begin{equation} \label{eq1.2}
b(x)=\left\{
\begin{array}{ll}
-\alpha x^{-q},&\qquad 0<x\leq M _1\\
\text{bounded measurable},&\qquad M_1<x<M_2\\
-\beta x^{-p},&\qquad M_2\leq x.
\end{array}
\right.
\end{equation}
By standard facts, the law of $X_t$ on the space of continuous paths in $(0,\infty)$ exists uniquely up to an explosion time. Given $a\in[0,\infty]$, let
\[
\tau_a(X)=\inf\{t>0:X_t=a\}.
\]
Below, we will show the explosion time is finite:
\begin{equation}\label{eq1.3}
P_x(\tau_0(X)<\infty)=1.
\end{equation}
Here is our main theorem.
\begin{thm}\label{thm1.1}
Under the condition \eqref{eq1.2}, the solution $X_t$ of \eqref{eq1.1} satisfies
\[
\lim_{t\to\infty}t^{-(1-p)/(1+p)}\log P_x(\tau_0(X)>t)=-\,\gamma(p,\beta),
\]
where
\begin{equation}\label{eq1.4}
\gamma(p,\beta)=\tfrac12p^{-2p/(1+p)}\beta^{2/(1+p)}\left[B\left(\tfrac12,\tfrac{1-p}{2p}\right)+B\left(\tfrac32,\tfrac{1-p}{2p}\right)\right]B\left(\tfrac12,\tfrac{1-p}{2p}\right)^{-(1-p)/(1+p)}.
\end{equation}
\end{thm}

\bigskip\n Our result extends to drifts of the form $\alpha(x)x^q\ell_1(x)$ near $0$ and $\beta(x)x^p\ell_2(x)$ near $\infty$, where $\ell_1$ and $\ell_2$ are slowly varying at $0$ and $\infty$, respectively, and the the coefficients $\alpha(x)$ and $\beta(x)$ have limiting values $\alpha_0$ and $\beta_0>0$. In the context of regular variation, one would say $x^q\ell_1(x)$ is regularly varying at $0$ from the right, with index $q$, and $x^p\ell_2(x)$ is regularly varying at infinity, with index $p$. 

\bigskip\n
To be precise, a measurable function $\ell>0$ is slowly varying at infinity if for each $\lambda>0$,
\[
\ell(\lambda x)/\ell(x)\to 1,\text{ as $x\to\infty$.}
\]
Similarly, measurable function $\ell>0$ is slowly varying at zero from the right if for each $\lambda>0$,
\[
\ell(\lambda x)/\ell(x)\to 1,\text{ as $x\to0^+$.}
\]

\begin{thm}\label{thm1.2}
Suppose $\alpha:(0,\infty)\to\mathbb{R}$ and $\beta:(0,\infty)\to(0,\infty)$ are continuous and $p,q\in(0,1)$. Assume:
\begin{itemize}
\item $\alpha(x)$ is bounded in a neighborhood of $0$ and $\lim_{x\to\infty}\beta(x)=\beta_0>0$ exists;
\item $\ell_1$ is continuous and slowly varying at zero from the right;
\item $\ell_2$ is continuous and slowly varying at infinity.
\end{itemize}
Suppose for some $0<M_1<M_2$,
\begin{equation} \label{eq1.5}
b(x)=\left\{
\begin{array}{ll}
-\alpha(x) x^q\ell_1(x),&\qquad 0<x\leq M _1\\
\text{\rm bounded measurable},&\qquad M_1<x<M_2\\
-\beta(x) x^p\ell_2(x),&\qquad M_2\leq x.
\end{array}
\right.
\end{equation}
Then the solution $X_t$ of \eqref{eq1.1} satisfies
\[
\lim_{t\to\infty}t^{-(1-p)/(1+p)}\log P_x(\tau_0(X)>t)=-\,\gamma(p,\beta_0),
\]
where $\gamma(p,\beta_0)$ is from \eqref{eq1.4}.

\end{thm}

\bigskip\n In the sequel, for a stochastic process $Z_t$ in $(0,\infty)$, for $0<\vp<M$, we will write
\[
\tau_{\vp,M}(Z)=\tau_\vp(Z)\wedge\tau_M(Z).
\]


\section{Preliminaries}\label{sec2}

First we prove zero is hit almost surely.

\begin{lem}\label{lem2.1}
For $X_t$ as in \eqref{eq1.1} with $b$ as in \eqref{eq1.2},
\[
P_x(\tau_0(X)<\infty)=1\text{ for all $x>0$.}
\]
\end{lem}
\n This will be a consequence of the following special case.

\begin{lem}\label{lem2.2}
For $X_t$ from \eqref{eq1.1}, where $b$ from \eqref{eq1.2} is continuous,
\[
P_x(\tau_0(X)<\infty)=1\text{ for all $x>0$.}
\]
\end{lem}

\begin{proof}
This is a consequence of the proof of a very general result (part (ii) in Theorem 1.1 of Chapter 5) in Pinsky (1995). It is a bit cumbersome to verify the hypotheses there, so we will present the proof in our very special case. Let
\[
L=\frac12\frac{d^2}{dx^2}+b(x)\frac{d}{dx}
\]
be the differential operator corresponding to $X_t$. For any $0<r_1<r_2$, write 
\[\tau_{r_1,r_2}=\tau_{r_1,r_2}(X).
\]
By Theorem 1.1 in Chapter 2 of Pinsky, $\tau_{r_1,r_2}<\infty$ a.s. $P_x$ for any $x\in(r_1,r_2)$. Define
\[
f(x)=\int_0^x\exp\left(-2\int_0^zb(y)\,dy \right)\,dz,\quad x>0
\]
(note the inner integral is finite because $0<q<1$). Then $f\in C^2(0,\infty)$ and $Lf=0$. By It\^o's formula, 
\[
E_x[f(X(\tau_{r_1,r_2}))]=f(x).
\]
This implies
\begin{equation}\label{eq2.1}
P_x(\tau_{r_1}<\tau_{r_2})=\frac{f(x)-f(r_2)}{f(r_1)-f(r_2)}.
\end{equation}
The key fact is that for large $r_2$,
\begin{align*}
f(r_2)&\geq\int_{M_2}^{r_2}\exp\left(-2\int_0^zb(y)\,dy \right)\,dz\\
&=\int_{M_2}^{r_2}\exp\left(-2\int_0^{M_2}b(y)\,dy -2\int_{M_2}^zb(y)\,dy\right)\,dz\\
&=\exp\left(-2\int_0^{M_2}b(y)\,dy\right) \int_{M_2}^{r_2}\exp\left(-2\int_{M_2}^zb(y)\,dy\right)\,dz\\
&=C\int_{M_2}^{r_2}\exp\left(2\beta\int_{M_2}^zy^{-p}\,dy\right)\,dz
\intertext{(where $C<\infty$, since $b$ is integrable on any interval $(0,M)$)}
&=C\int_{M_2}^{r_2}\exp\left(\tfrac{2\beta}{1-p}\left(z^{1-p}-M_2^{1-p}\right)\right)\,dz.
\end{align*}
Since $1-p$ and $\tfrac{2\beta}{1-p}$ are both positive, we see $f(r_2)\to\infty$ as $r_2\to\infty$. Using this in \eqref{eq2.1} gives $P_x(\tau_{r_1}<\infty)=1$, and then letting $r_1\downarrow 0$ gives $P_x(\tau_0<\infty)=1$, as claimed.
\end{proof}

\noindent{\bf Proof of Lemma \ref{lem2.1}.} Given $b$ as in \eqref{eq1.2}, since $b$ is bounded on $[M_1,M_2]$, we can choose $\widetilde{M_1}<M_1$ and $\widetilde{M_2}>M_2$, along with a corresponding continuous $\tilde b$ satisfying \eqref{eq1.2} for $\widetilde{M_1}$, $\widetilde{M_2}$, such that $b\leq \tilde b$ on $(0,\infty)$. Then for $d\widetilde X_t=dB_t+\tilde b(\widetilde X_t)\,dt$, $\widetilde X_0=x$, by the Comparision Theorem for SDEs (Ikeda and Watanabe (1989), Theorem 1.1 in Chapter VI), $X_t\leq\widetilde X_t$ almost surely and so
\[
1=P_x(\tau_0(\widetilde X)<\infty)\leq P_x(\tau_0(X)<\infty),
\]
and so $P_x(\tau_0(X)<\infty)=1$, as claimed.\hfill$\square$

\bigskip
\n The next results pertain to the case when $b\in C^3(0,\infty)$.

\begin{lem}\label{lem2.3}
Suppose $b$ from \eqref{eq1.2} is in $C^1$. Then for
\begin{equation}\label{eq2.2}
h(x)=\exp\left(\int_0^xb(y)\,dy\right), \,x>0,
\end{equation}

\begin{equation}\label{eq2.3}
\text{$h(x)$ is bounded on $[M_1,M_2]$;}
\end{equation}

\begin{equation}\label{eq2.4}
\text{$h(x)$ is bounded below away from $0$ on $[M_1,M_2]$;}
\end{equation}

\begin{equation}\label{eq2.5}
h(x)=\exp\left(-\tfrac{\alpha}{1-q}x^{1-q}\right)\text{on $(0,M_1]$;}
\end{equation}
for some positive constant $C$,
\begin{equation}\label{eq2.6}
h(x)=C\exp\left(-\tfrac{\beta}{1-p}x^{1-p}\right), \,x>M_2.
\end{equation}
\end{lem}
\begin{proof}
For some constants $C_1$ and $C_2$,
\begin{align}\label{eq2.7}
\int_0^x b(y)\,dy&=-\tfrac{\alpha}{1-q}x^{1-q}I_{(0,M_1]}(x)+\left[C_1+\int_{M_1}^xb(y)\,dy\right]I_{(M_1,M_2)}(x)\notag\\
&+\left[C_2-\tfrac{\beta}{1-p}x^{1-p}\right]I_{[M_2,\infty)}(x).
\end{align}
The assertions of the Lemma follow immediately from this.
\end{proof}

\begin{lem}\label{lem2.4}
Suppose $b$ from \eqref{eq1.2} is in $C^1$ and $\alpha\geq0$. Then for 
\begin{equation}\label{eq2.8}
V=-\tfrac12(b^2+b'),
\end{equation}
we have $V\leq VI_{(M_1,\infty)}$.
\end{lem}
\begin{proof}
Since $\alpha\geq 0$ and $0<q<1$, on the interval $(0,M_1]$ we have
\[
V(x)=-\frac{\alpha^2}{2x^{2q}}-\frac{\alpha q}{2x^{q+1}}\leq0.
\]
The conclusion is immediate.
\end{proof}

\n The next result we will need follows from the formula (1.1) in the proof of Theorem 1.1 from Section 4.1 in the book by Pinsky (1995). To satisfy the hypotheses he requires, we take the drift to be in $C^3$.

\begin{lem}\label{lem2.5}
Suppose $X_t$ is from \eqref{eq1.1} in the introduction, where $X_0=x>0$ and $b$ given by \eqref{eq1.2} is in $C^3$. Then for $h$ from \eqref{eq2.2} and $V$ from \eqref{eq2.8},
\[
P_x(\tau_0(X)>t)=E_x\left[\exp\left(\int_0^tV(B_s)\,ds\right)h(B_t)I_{\tau_0(B)>t}\right].
\]
\end{lem}
\begin{proof}
Let $L=\tfrac{1}{2}\tfrac{d^2}{dx^2}+V$. Note $h$ is positive and $L$-harmonic on $(0,\infty)$. The $h$-transform $L^h$ of $L$ is defined by
\[
L^hf=\frac1hL(hf).
\]
It is easy to show that
\[
L^h=\frac{1}{2}\frac{d^2}{dx^2}+b\frac{d}{dx}.
\]
Let $0<\vp<M<\infty$ be such that $x\in(\vp,M)$. Since $V$ is bounded on $[\vp,M]$, the formula (1.1) in the proof of Theorem 1.1 in Section 4.1 of Pinsky (1995) holds: for each continuous $f$ with compact support in $(\vp,M)$,
\begin{align*}
\frac{1}{h(x)}E_x&\left[\exp\left(\int_0^tV(B_s)\,ds\right)(hf)(B_t)I(\tau_{\vp,M}(B)>t)\right]\\
&=E_x\left[\exp\left(\int_0^t\frac{Lh}{h}(X_s)\,ds\right)f(X_t)I(\tau_{\vp,M}(X)>t)\right]\\
&=E_x\left[f(X_t)I(\tau_{\vp,M}(X)>t)\right],
\end{align*}
since $Lh=0$. From this we get
\[
\frac{1}{h(x)}E_x\left[\exp\left(\int_0^tV(B_s)\,ds\right)h(B_t)I(\tau_{\vp,M}(B)>t)\right]=P_x(\tau_{\vp,M}(X)>t).
\]
By Monotone Convergence and Lemma \ref{lem2.1}, we can let $\vp\downarrow 0$ and $M\uparrow \infty$ to get
\[
\frac{1}{h(x)}E_x\left[\exp\left(\int_0^tV(B_s)\,ds\right)h(B_t)I_{\tau_0(B)>t}\right]=P_x(\tau_0(X)>t),
\]
as claimed.
\end{proof}


\section{Lower Bound}\label{sec3}

The main result of this section is the following lower bound.

\begin{thm}\label{thm3.1} Let $X_t$ be as in \eqref{eq1.1}, where $b$ is from \eqref{eq1.2}. Then
\[
\liminf_{t\to\infty}t^{-(1-p)/(1+p)}\log P_x(\tau_0(X)>t)\geq-\,\gamma(p,\beta),
\]
where $\gamma(p,\beta)$ is from \eqref{eq1.4}.

\end{thm}

\n We first prove a special case.

\begin{thm}\label{thm3.2} Let $X_t$ be as in \eqref{eq1.1}, where $b$ from \eqref{eq1.2} satisfies $\alpha\geq 0$ and $b\in C^3$. Then
\[
\liminf_{t\to\infty}t^{-(1-p)/(1+p)}\log P_x(\tau_0(X)>t)\geq-\,\gamma(p,\beta),
\]
where $\gamma(p,\beta)$ is from \eqref{eq1.4}.
\end{thm}

\n The proof of the theorem is long, so we break it up into pieces. By Lemma \ref{lem2.5}
\begin{equation}\label{eq3.1}
P_x(\tau_0(X)>t)=E_x\left[\exp\left(\int_0^tV(B_s)\,ds\right)h(B_t)I_{\tau_0(B)>t}\right].
\end{equation}
Define
\[
K_0=\{\omega\in C_0:\int_0^1|\dot\omega_u|^2\,du<\infty\}
\]
and
\begin{equation}\label{eq3.2}
F(\omega)=\frac{\beta^2}{2}\int_0^1|\omega_u|^{-2p}\,du+\frac{\beta}{1-p}|\omega_1|^{1-p}+\frac12\int_0^1\dot\omega_u^2\,du,\quad \omega\in K_0.
\end{equation}
Let $g\in K_0$ with $g\geq0$. Let $\delta>0$ be given and set $\tilde g=g+\delta$,
\begin{equation}\label{eq3.3}
\vp=t^{-(1-p)/(1+p)},
\end{equation}
and
\begin{equation}\label{eq3.4}
Z_u=B_u-\tilde g(u)/\sqrt\vp.
\end{equation}
For $Z_0>0$, we have
\begin{equation}\label{eq3.5}
\tau_0(Z)>1\Longrightarrow \tau_0(B)>1
\end{equation}
\begin{equation}\label{eq3.6}
\tau_0(Z)>1\Longrightarrow B_u\geq\tilde g(u)/\sqrt\vp,\,u\in[0,1].
\end{equation}
Now scale and use \eqref{eq3.5}:
\begin{align}\label{eq3.7}
P_x(\tau_0(X)>t)&=E_x\left[\exp\left(\vphantom{\int_0^t}\right.\right.\left.\left.\int_0^tV(B_s)\,ds\right)h(B_t)I_{\tau_0(B)>t}\right]\notag\\
&=E_{x/\sqrt t}\left[\exp\left(t\int_0^1V(\sqrt tB_u)\,du\right)h(\sqrt tB_1)I_{\tau_0(B)>1}\right]\notag\\
&\geq E_{x/\sqrt t}\left[\exp\left(t\int_0^1V(\sqrt tB_u)\,du\right)h(\sqrt tB_1)I_{\tau_0(Z)>1}\right].
\end{align}
\begin{lem}\label{lem3.3}
For $\tau_0(Z)>1$ and $\vp$ from \eqref{eq3.3},
\[
t\int_0^1V(\sqrt tB_u)\left[I_{\sqrt tB_u<M_1}+I_{M_2<\sqrt tB_u}\right]\,du\geq G_1(\tilde g,\vp),
\]
where
\[
\lim_{\vp\to 0}\vp\, G_1(\tilde g,\vp)= -\frac{\beta^2}{2}\int_0^1\tilde g(u)^{-2p}\,du.
\]
\end{lem}
\begin{proof}
By \eqref{eq2.8},
\[
V(x)\left[I_{0<x<M_1}+I_{M_2<x}\right]=-\left(\frac{\alpha^2}{2x^{2q}}+\frac{q\alpha}{2x^{q+1}}\right)I_{0<x<M_1}-\left(\frac{\beta^2}{2x^{2p}}+\frac{p\beta}{2x^{p+1}}\right)I_{M_2<x}.
\]
This is nonpositive and increasing, since $\alpha\geq 0$ and $\beta>0$. By \eqref{eq3.6}, for $u\in[0,1]$, $B_u\geq\tfrac{\tilde g(u)}{\sqrt\vp}$, and so
$V\left(\sqrt tB_u\right)\geq V\left(\sqrt{\tfrac{t}{\vp}}\,\tilde g(u)\right)$. Thus for $\tau_0(Z)>1$,
\begin{align}\label{eq3.8}
t\int_0^1V(\sqrt tB_u)&\left[I_{\sqrt tB_u<M_1}+I_{M_2<\sqrt tB_u}\right]\,du\geq t\int_0^1V\left(\sqrt{\tfrac{t}{\vp}}\,\tilde g(u)\right)\left[I_{\sqrt tB_u<M_1}+I_{M_2<\sqrt tB_u}\right]\,du\notag\\
&=-t\int_0^1\left[\frac{\alpha^2}{2}\left(\sqrt{\tfrac{t}{\vp}}\,\tilde g(u)\right)^{-2q}+\frac{q\alpha}{2}\left(\sqrt{\tfrac{t}{\vp}}\,\tilde g(u)\right)^{-q-1}\right]I_{\sqrt tB_u<M_1}\,du\notag\\
&\qquad-t\int_0^1\left[\frac{\beta^2}{2}\left(\sqrt{\tfrac{t}{\vp}}\,\tilde g(u)\right)^{-2p}+\frac{p\beta}{2}\left(\sqrt{\tfrac{t}{\vp}}\,\tilde g(u)\right)^{-p-1}\right]I_{M_2<\sqrt tB_u}\,du\notag\\
&\geq-t\int_0^1\left[\frac{\alpha^2}{2}\left(\sqrt{\tfrac{t}{\vp}}\,\tilde g(u)\right)^{-2q}+\frac{q\alpha}{2}\left(\sqrt{\tfrac{t}{\vp}}\,\tilde g(u)\right)^{-q-1}\right]I_{\sqrt {t/\vp}\tilde g(u)<M_1}\,du\notag\\
&\qquad-t\int_0^1\left[\frac{\beta^2}{2}\left(\sqrt{\tfrac{t}{\vp}}\,\tilde g(u)\right)^{-2p}+\frac{p\beta}{2}\left(\sqrt{\tfrac{t}{\vp}}\,\tilde g(u)\right)^{-p-1}\right]\,du\notag\\
\intertext{(since $I_{M_2<\sqrt tB_u}\leq 1$ and since $\tau_0(Z)>1$ implies $I_{\sqrt tB_u<M_1}=I_{\sqrt t(Z_u+\tilde g(u)/\sqrt \vp)<M_1}\leq I_{\sqrt {t/\vp}\tilde g(u)<M_1}$ for $u\in[0,1]$)}
&=-\frac12\int_0^1\left[\alpha^2\vp^{(2q-p-1)/(1-p)}\tilde g(u)^{-2q}+q\alpha\vp^{(q-p)/(1-p)}\tilde g(u)^{-q-1}\right]I_{\tilde g(u)<M_1\vp^{1/(1-p)}}\,du\notag\\
&\qquad-\frac12\int_0^1\left[\beta^2\vp^{-1}\tilde g(u)^{-2p}+p\beta\tilde g(u)^{-p-1}\right]\,du,
\end{align}
where we have used \eqref{eq3.3} to replace the variable $t$ in terms of the variable $\vp$.
Define  
\[
G_1(\tilde g,\vp)=\text{RHS\eqref{eq3.8}.}
\]
Then
\begin{align}\label{eq3.9}
\vp\,G_1(\tilde g,\vp)&=-\frac12\int_0^1\left[\alpha^2\vp^{2(q-p)/(1-p)}\tilde g(u)^{-2q}+q\alpha\vp^{(1+q-2p)/(1-p)}\tilde g(u)^{-q-1}\right]I_{\tilde g(u)<M_1\vp^{1/(1-p)}}\,du\notag\\
&\qquad-\frac12\int_0^1\left[\beta^2\tilde g(u)^{-2p}+p\beta\vp\tilde g(u)^{-p-1}\right]\,du.
\end{align}
Now if $\vp<(\delta/M_1)^{1-p}$, then for any $u\in[0,1]$, $\tilde g(u)\geq \delta >M_1\vp^{1/(1-p)}$. Thus the first integral in \eqref{eq3.9} is $0$. The integrand in the second integral is bounded because $\tilde g\geq\delta$ and $p>0$. It follows that
\[
\lim_{\vp\to0}\vp\,G_1(\tilde g,\vp)=-\frac12\int_0^1\beta^2\tilde g(u)^{-2p}\,du,
\]
as claimed.
\end{proof} 

\n Having studied 
\[
t\int_0^1V(\sqrt tB_u)\left[I_{\sqrt tB_u<M_1}+I_{M_2<\sqrt tB_u}\right]\,du,
\] 
we now turn attention to 
\[t\int_0^1V(\sqrt tB_u)\left[I_{M_1\leq\sqrt tB_u\leq M_2}\right]\,du.
\]
\begin{lem}\label{lem3.4}
For any positive $C_1$ and $\gamma$, as $\vp\to 0$,
\begin{align*}
P_{x/\sqrt t}&\left(C_1t\int_0^1I_{\sqrt tB_u+\sqrt{t/\vp}\,\tilde g(u)<M_2}\,du+\frac{\beta}{1-p}\left(\sqrt tB_1\right)^{1-p}+\frac{1}{\sqrt\vp}\int_0^1\tilde g'(u)\,dB_u>\frac\gamma\vp,\,\tau_0(B)>1\right)\\
&\qquad=o\left(P_{x/\sqrt t}\left(\tau_0(B)>1\right)\right).
\end{align*}
\end{lem}
\begin{proof}
Since $\tau_0(B)>1$ implies $B_u>0$ for $u\in[0,1]$,
\begin{align}\label{eq3.10}
P_{x/\sqrt t}&\left(C_1t\int_0^1I_{\sqrt tB_u+\sqrt{t/\vp}\,\tilde g(u)<M_2}\,du>\frac\gamma{3\vp},\,\tau_0(B)>1\right)\notag\\
&\leq e^{-\gamma/3\vp}E_{x/\sqrt t}\left[\exp\left(C_1t\int_0^1I_{\sqrt tB_u+\sqrt{t/\vp}\,\tilde g(u)<M_2}\,du\right)I_{\tau_0(B)>1}\right]\notag\\
&\leq e^{-\gamma/3\vp}E_{x/\sqrt t}\left[\exp\left(C_1t\int_0^1I_{\sqrt{t/\vp}\,\tilde g(u)<M_2}\,du\right)I_{\tau_0(B)>1}\right]\notag\\
&\leq e^{-\gamma/3\vp}\exp\left(C_1t\int_0^1I_{\sqrt{t/\vp}\,\delta<M_2}\,du\right)P_{x/\sqrt t}(\tau_0(B)>1),
\end{align}
(since $\tilde g\geq\delta$).
But by \eqref{eq3.3}, 
\[
\dfrac t\vp=\vp^{-(1+p)/(1-p)-1}\to\infty\text{ as } \vp\to 0,
\]
so for small $\vp$, $I_{\sqrt{t/\vp}\,\delta<M_2}=0$ and we get
\begin{equation}\label{eq3.11}
P_{x/\sqrt t}\left(C_1t\int_0^1I_{\sqrt tB_u+\sqrt{t/\vp}\,\tilde g(u)<M_2}\,du>\frac\gamma{3\vp},\,\tau_0(B)>1\right)\leq e^{-\gamma/3\vp}.
\end{equation}
Next, for some constant $C$ (whose exact value might change from line to line), independent of $\vp$,
\begin{align}\label{eq3.12}
P_{x/\sqrt t}&\left(\frac{\beta}{1-p}\left(\sqrt tB_1\right)^{1-p}>\frac\gamma{3\vp}\right)=P_{x/\sqrt t}\left(B_1>C\left(\frac\gamma{\vp}\right)^{1/(1-p)}t^{-1/2}\right)\notag\\
&\leq \exp\left(-C\vp^{-1/(1-p)}t^{-1/2}\right)E_{x/\sqrt t}\left[e^{B_1}\right]\notag\\
&=\exp\left(-C\vp^{-1/2}+x\vp^{(1+p)/2(1-p)}+1/2\right),
\end{align}
using \eqref{eq3.3} to write $t$ in terms of $\vp$. 

\n Finally,
\begin{align}\label{eq3.13}
P_{x/\sqrt t}&\left(\frac{1}{\sqrt\vp}\int_0^1\tilde g'(u)\,dB_u>\frac\gamma{3\vp}\right)\leq e^{-\gamma/3\sqrt\vp}\,E_{x/\sqrt t}\left[\exp\left(\int_0^1\tilde g'(u)\,dB_u\right)\right]\notag\\
&=e^{-\gamma/3\sqrt\vp}\,\exp\left(\frac12\int_0^1\tilde g'(u)^2\,du\right),
\end{align}
since under $P_{x/\sqrt t}$, $\int_0^1\tilde g'(u)\,dB_u$ is normal with mean $0$ and variance $\int_0^1\tilde g'(u)^2\,du$. Combining \eqref{eq3.11}--\eqref{eq3.13}, we get
\begin{align}\label{eq3.14}
P_{x/\sqrt t}&\left(C_1t\int_0^1I_{\sqrt tB_u+\sqrt{t/\vp}\,\tilde g(u)<M_2}\,du+\frac{\beta}{1-p}\left(\sqrt tB_1\right)^{1-p}+\frac{1}{\sqrt\vp}\int_0^1\tilde g'(u)\,dB_u>\frac\gamma\vp,\,\tau_0(B)>1\right)\notag\\
&\leq P_{x/\sqrt t}\left(C_1t\int_0^1I_{\sqrt tB_u+\sqrt{t/\vp}\,\tilde g(u)<M_2}\,du>\frac{\gamma}{3\vp},\,\tau_0(B)>1\right)\notag\\
&\qquad +P_{x/\sqrt t}\left(\frac{\beta}{1-p}\left(\sqrt tB_1\right)^{1-p}>\frac{\gamma}{3\vp},\,\tau_0(B)>1\right)\notag\\
&\qquad\qquad +P_{x/\sqrt t}\left(\frac{1}{\sqrt\vp}\int_0^1\tilde g'(u)\,dB_u>\frac{\gamma}{3\vp},\,\tau_0(B)>1\right)\notag\\
&\leq e^{-\gamma/3\vp}+C_2\exp\left(-C\vp^{-1/2}+x\vp^{(1+p)/2(1-p)}\right)+C_3e^{-\gamma/3\sqrt\vp},
\end{align}
where the constants are independent of $\vp$. From Feller (1971), as $t\to\infty$ (or equivalently, as $\vp\to 0$)
\begin{align}\label{eq3.15}
P_{x/\sqrt t}&\left(\tau_0(B)>1\right)=\frac{2}{\sqrt{2\pi}}\,\int_0^{x/\sqrt t}e^{-u^2/2}\,du\sim\frac{2}{\sqrt{2\pi}}\frac{x}{\sqrt t}\notag\\
&=C_4\,\vp^{(1+p)/2(1-p)},
\end{align}
where $C_4$ is independent of $\vp$ and we have used \eqref{eq3.3} to write $t$ in terms of $\vp$. Combining \eqref{eq3.14}--\eqref{eq3.15} yields the conclusion of the Lemma.
\end{proof}
\begin{lem}\label{lem3.5}
For any $C_1>0$ and $\gamma>0$, 
\begin{align*}
\lim_{\vp\to 0}\vp\,\log P_{x/\sqrt t}&\left(C_1t\int_0^1I_{\sqrt tB_u+\sqrt{t/\vp}\tilde g(u)<M_2}\,du+\frac{\beta}{1-p}\left(\sqrt tB_1\right)^{1-p}\right.\\
&\left. \qquad +\frac{1}{\sqrt\vp}\int_0^1\tilde g'(u)\,dB_u\leq\frac{\gamma}{\vp},\tau_0(B)>1\right)=0.
\end{align*}
\end{lem}
\begin{proof}
Write
\[
M=C_1t\int_0^1I_{\sqrt tB_u+\sqrt{t/\vp}\tilde g(u)<M_2}\,du+\frac{\beta}{1-p}\left(\sqrt tB_1\right)^{1-p}+\frac{1}{\sqrt\vp}\int_0^1\tilde g'(u)\,dB_u.
\]
Then we want to show
\begin{equation}\label{eq3.16}
\lim_{\vp\to 0}\vp\,\log P_{x/\sqrt t}\left(M\leq\frac\gamma\vp,\,\tau_0(B)>1\right)=0.
\end{equation}
Notice
\begin{align*}
P_{x/\sqrt t}\left(M\leq\frac\gamma\vp,\right.&\left.\,\tau_0(B)>1\vphantom{\frac\gamma\vp}\right)=P_{x/\sqrt t}\left(\tau_0(B)>1\right)-P_{x/\sqrt t}\left(M>\frac\gamma\vp,\,\tau_0(B)>1\right)\\
&=P_{x/\sqrt t}\left(\tau_0(B)>1\right)\left[1-P_{x/\sqrt t}\left(M>\frac\gamma\vp,\,\tau_0(B)>1\right)\bigg\slash P_{x/\sqrt t}\left(\tau_0(B)>1\right)\right].
\end{align*}
It follows that
\begin{align*}
\vp\,\log P_{x/\sqrt t}&\left(M\leq\frac\gamma\vp,\,\tau_0(B)>1\right)=\vp\,\log P_{x/\sqrt t}\left(\tau_0(B)>1\right)\\
&+\vp\,\log\left[1-P_{x/\sqrt t}\left(M>\frac\gamma\vp,\,\tau_0(B)>1\right)\bigg\slash P_{x/\sqrt t}\left(\tau_0(B)>1\right)\right].
\end{align*}
Equation \eqref{eq3.16} follows from this, using Lemma \ref{lem3.4} and \eqref{eq3.15}.
\end{proof}

\begin{lem}\label{lem3.6}
For any $C_1>0$,
\begin{align*}           
\liminf_{\vp\to 0}\vp\,&\log E_{x/\sqrt t}\left[\exp\left(-C_1t\int_0^1I_{\sqrt tB_u+\sqrt{t/\vp}\tilde g(u)<M_2}\,du-\frac{\beta}{1-p}\left(\sqrt tB_1\right)^{1-p}\right.\right.\\
&\left.\left. -\frac{1}{\sqrt\vp}\int_0^1\tilde g'(u)\,dB_u\right)I_{\tau_0(B)>1}\right]\geq 0.
\end{align*}
\end{lem}
\begin{proof}
For any $\gamma>0$, 
\begin{align*}           
\liminf_{\vp\to 0}\vp\,&\log E_{x/\sqrt t}\left[\exp\left(-C_1t\int_0^1I_{\sqrt tB_u+\sqrt{t/\vp}\tilde g(u)<M_2}\,du-\frac{\beta}{1-p}\left(\sqrt tB_1\right)^{1-p}\right.\right.\\
&\left.\left.\qquad -\frac{1}{\sqrt\vp}\int_0^1\tilde g'(u)\,dB_u\right)I_{\tau_0(B)>1}\right]\\
&\geq \liminf_{\vp\to 0}\vp\,\log E_{x/\sqrt t}\left[e^{-\gamma/\vp}I\left(C_1t\int_0^1I_{\sqrt tB_u+\sqrt{t/\vp}\tilde g(u)<M_2}\,du+\frac{\beta}{1-p}\left(\sqrt tB_1\right)^{1-p}\right.\right.\\
&\left.\left. \qquad+\frac{1}{\sqrt\vp}\int_0^1\tilde g'(u)\,dB_u\leq\frac{\gamma}{\vp}\right)I_{\tau_0(B)>1}\right]\\
&=\liminf_{\vp\to 0}\left[-\gamma+\vp\,\log P_{x/\sqrt t}\left(C_1t\int_0^1I_{\sqrt tB_u+\sqrt{t/\vp}\tilde g(u)<M_2}\,du+\frac{\beta}{1-p}\left(\sqrt tB_1\right)^{1-p}\right.\right.\\
&\left.\left. \qquad+\frac{1}{\sqrt\vp}\int_0^1\tilde g'(u)\,dB_u\leq\frac{\gamma}{\vp},\tau_0(B)>1\right)\right]\\
&\geq -\,\gamma+0,
\end{align*}
by Lemma \ref{lem3.5}. Let $\gamma\to0$ to finish.
\end{proof}

\begin{lem}\label{lem3.7}
We have
\begin{align*}
\liminf_{\vp\to 0}\vp\,\log E_{x/\sqrt t}&\left[\exp\left(t\int_0^1V(\sqrt tB_u)I_{M_1<\sqrt tB_u<M_2}\,du\right)h(\sqrt tB_1)I_{\tau_0(Z)>1}\right]\\
&\geq -\frac{\beta}{1-p}\,\tilde g(1)^{1-p}-\frac12\int_0^1\tilde g'(u)^2\,du.
\end{align*}
\end{lem}

\begin{proof}
Since $V$ is bounded on $[M_1,M_2]$, there is $C_1>0$ such that $|V|\leq C_1$ on that interval. Then
\begin{align}\label{eq3.17}
E_{x/\sqrt t}&\left[\exp\left(t\int_0^1V(\sqrt tB_u)I_{M_1<\sqrt tB_u<M_2}\,du\right)h(\sqrt tB_1)I_{\tau_0(Z)>1}\right]\notag\\
&\geq E_{x/\sqrt t}\left[\exp\left(t\int_0^1V(\sqrt tB_u)I_{M_1<\sqrt tB_u<M_2}\,du\right)h(\sqrt tB_1)I_{M_2<\sqrt tB_1}I_{\tau_0(Z)>1}\right]\notag\\
&\geq CE_{x/\sqrt t}\left[\exp\left(-C_1t\int_0^1I_{\sqrt tB_u<M_2}\,du-\frac{\beta}{1-p}(\sqrt tB_1)^{1-p}\right)I_{M_2<\sqrt tB_1}I_{\tau_0(Z)>1}\right]\notag\\
\intertext{(where $C$ is from \eqref{eq2.6} and using that $I_{M_1<\sqrt tB_u<M_2}\leq I_{\sqrt tB_u<M_2}$)}
&\geq CE_{x/\sqrt t}\left[\exp\left(-C_1t\int_0^1I_{\sqrt tZ_u+\sqrt{t/\vp}\tilde g(u)<M_2}\,du-\frac{\beta}{1-p}\left(\sqrt tZ_1+\sqrt{t/\vp}\tilde g(1)\right)^{1-p}\right)\cdot\right.\notag\\
&\qquad\left.\vphantom{\int_0^1}\cdot I_{M_2<\sqrt {t/\vp}\tilde g(1)}I_{\tau_0(Z)>1}\right]\notag\\
\intertext{\big(using \eqref{eq3.4} and that $I_{M_2<\sqrt tB_1}= I_{M_2<\sqrt tZ_1+\sqrt{t/\vp}\tilde g(1)}\geq I_{M_2<\sqrt {t/\vp}\tilde g(1)}$ for $\tau_0(Z)>1$\big)}
&\geq CE_{x/\sqrt t}\left[\exp\left(-C_1t\int_0^1I_{\sqrt tZ_u+\sqrt{t/\vp}\tilde g(u)<M_2}\,du-\frac{\beta}{1-p}\left(\sqrt tZ_1\right)^{1-p}\right.\right.\notag\\
&\qquad\left.\left. -\frac{\beta}{1-p}\left(\sqrt{t/\vp}\,\tilde g(1)\right)^{1-p}\right)I_{M_2<\sqrt {t/\vp}\tilde g(1)}I_{\tau_0(Z)>1}\right]\notag\\
\intertext{(using that $(a+b)^{1-p}\leq a^{1-p}+b^{1-p}$ since $0<p<1$)}
&=C\exp\left(-\frac{\beta}{1-p}\left(\sqrt{t/\vp}\,\tilde g(1)\right)^{1-p}\right)I_{M_2<\sqrt {t/\vp}\tilde g(1)}\cdot\notag\\
&\qquad\cdot E_{x/\sqrt t}\left[\exp\left(-C_1t\int_0^1I_{\sqrt tB_u+\sqrt{t/\vp}\tilde g(u)<M_2}\,du-\frac{\beta}{1-p}\left(\sqrt tB_1\right)^{1-p}\right.\right.\notag\\
&\qquad\qquad\left.\left. -\frac{1}{\sqrt\vp}\int_0^1\tilde g'(u)\,dB_u-\frac{1}{2\vp}\int_0^1\tilde g'(u)^2\,du\right)I_{\tau_0(B)>1}\right]\notag\\
\intertext{(by the Cameron-Martin-Girsanov Theorem)}
&=C\exp\left(-\frac{\beta}{1-p}\left(\sqrt{t/\vp}\,\tilde g(1)\right)^{1-p}-\frac{1}{2\vp}\int_0^1\tilde g'(u)^2\,du\right)I_{M_2<\sqrt {t/\vp}\tilde g(1)}\cdot\notag\\
&\qquad\cdot E_{x/\sqrt t}\left[\exp\left(-C_1t\int_0^1I_{\sqrt tB_u+\sqrt{t/\vp}\tilde g(u)<M_2}\,du-\frac{\beta}{1-p}\left(\sqrt tB_1\right)^{1-p}\right.\right.\notag\\
&\qquad\qquad\left.\left. -\frac{1}{\sqrt\vp}\int_0^1\tilde g'(u)\,dB_u\right)I_{\tau_0(B)>1}\right].
\end{align}
From this and Lemma \ref{lem3.6}, we get
\begin{align*}
\liminf_{\vp\to 0}\vp\,&\log E_{x/\sqrt t}\left[\exp\left(t\int_0^1V(\sqrt tB_u)I_{M_1<\sqrt tB_u<M_2}\,du\right)h(\sqrt tB_1)I_{\tau_0(Z)>1}\right]\\
&\geq \liminf_{\vp\to 0}\vp\,\left(\log C-\frac{\beta}{1-p}\left(\sqrt{t/\vp}\,\tilde g(1)\right)^{1-p}-\frac{1}{2\vp}\int_0^1\tilde g'(u)^2\,du+\log I_{M_2<\sqrt {t/\vp}\tilde g(1)}\right)+0\\
&=\liminf_{\vp\to 0}\left(\vp\,\log C-\frac{\beta}{1-p}\tilde g(1)^{1-p}-\frac{1}{2}\int_0^1\tilde g'(u)^2\,du+\vp\,\log I_{M_2<\vp^{-1/(1+p)}\tilde g(1)}\right)\\
\intertext{(using \eqref{eq3.3} to substitute for $t$ in terms of $\vp$)}
&= -\frac{\beta}{1-p}\,\tilde g(1)^{1-p}-\frac12\int_0^1\tilde g'(u)^2\,du,
\end{align*}
as desired, where we have used that 
\[
\lim_{\vp\to 0}\,I_{M_2<\vp^{-1/(1+p)}\tilde g(1)}=1,
\]
since $\tilde g(1)=g(1)+\delta>0$.
\end{proof}

\noindent{\bf Proof of Theorem \ref{thm3.2}.} It was shown in DeBlassie and Smits (2007) that
\[
\gamma(p,\beta)=\inf_{\omega\in K_0}F(\omega)=\mathop{\inf_{\omega\in K_0}}_{\omega\geq 0}F(\omega),
\]
where $F$ is from \eqref{eq3.2} (see the very end of Section 5 and Theorem 6.1 in that article). Hence to prove Theorem \ref{thm3.2} we need to show
\[
\liminf_{t\to\infty}t^{-(1-p)/(1+p)}\log P_x(\tau_0(X)>t)\geq-\mathop{\inf_{\omega\in K_0}}_{\omega\geq 0}F(\omega).
\]
Now by \eqref{eq3.3} and \eqref{eq3.7},
\begin{align*}
\liminf_{t\to\infty}\, &t^{-(1-p)/(1+p)}\log P_x(\tau_0(X)>t)\\
&\geq\liminf_{\vp\to 0}\vp\,\log E_{x/\sqrt t}\left[\exp\left(t\int_0^1V(\sqrt tB_u)\,du\right)h(\sqrt tB_1)I_{\tau_0(Z)>1}\right]\\
&\geq\liminf_{\vp\to 0}\vp\,\log E_{x/\sqrt t}\left[\exp\left(G_1(\tilde g,\vp)+t\int_0^1V(\sqrt tB_u)I_{M_1<\sqrt tB_u<M_2}\,du\right)h(\sqrt tB_1)I_{\tau_0(Z)>1}\right]\\
\intertext{(by Lemma \ref{lem3.3})}
&=\liminf_{\vp\to 0}\left[\vp\,G_1(\tilde g,\vp)+\vp\,\log E_{x/\sqrt t}\left[\exp\right.\left(t\int_0^1V(\sqrt tB_u)I_{M_1<\sqrt tB_u<M_2}\,du\right)h(\sqrt tB_1)I_{\tau_0(Z)>1}\right]\\
&\geq -\frac{\beta^2}{2}\int_0^1\tilde g(u)^{-2p}\,du-\frac{\beta}{1-p}\,\tilde g(1)^{1-p}-\frac12\int_0^1\tilde g'(u)^2\,du.
\intertext{(by Lemmas \ref{eq3.2} and \ref{lem3.7}).}
\end{align*}
Recalling $\tilde g=g+\delta$, where $\delta >0$ was arbitrary, we can let $\delta\to 0$ to end up  with
\begin{align*}
\liminf_{t\to\infty}\,t^{-(1-p)/(1+p)}&\log P_x(\tau_0(X)>t)\geq -\frac{\beta^2}{2}\int_0^1 g(u)^{-2p}\,du-\frac{\beta}{1-p}\,g(1)^{1-p}-\frac12\int_0^1g'(u)^2\,du.\\
&=-F(g)
\end{align*}
Since $g\geq 0$ in $K_0$ was arbitrary, the desired lower bound in Theorem \ref{thm3.2} holds.\hfill$\square$

\bigskip
\noindent{\bf Proof of Theorem \ref{thm3.1}.} Since $b$ is bounded below on $[M_1,M_2]$, we can choose $C<0$ such that $C\leq b$ on $[M_1,M_2]$. Then for some $x_1<M_1$ and $x_2>M_2$, we can choose $\tilde b\in C^3(0,\infty)$ such that
\[
\tilde b(x)=\left\{\begin{array}{ll}
-|\alpha| x^{-q},&\quad x\in(0,x_1]\\
-\beta x^{-p},&\quad x\in [x_2,\infty),
\end{array}\right.
\]
and $b\geq \tilde b$ on $(0,\infty)$. By the Comparison Theorem, if $dY_t=dB_t+\tilde b(Y_t),\, Y_0=x$, then 
\[
P_x(\tau_0(X)>t)\geq P_x(\tau_0(Y)>t).
\]
The desired lower bound follows from Theorem \ref{thm3.2} applied to $Y$. \hfill$\square$


\section{Upper Bound: A Special $C^3$ Case}\label{sec4}

The main theorem of this section is the following special case. It will be crucial in proving the general case.
\begin{thm}\label{thm4.1}
Let $X_t$ be from \eqref{eq1.1} where $b$ from \eqref{eq1.2} satisfies $\alpha\geq 0$, $b\in C^3$ and $b'\geq 0$ on $(0,\infty)$. Then
\[
\limsup_{t\to\infty}\log P_x(\tau_0(X)>t)\leq-\,\gamma(p,\beta),
\]
where $\gamma(p,\beta)$ is from \eqref{eq1.4}.
\end{thm}
\begin{proof}
Let $h$ be from \eqref{eq2.2} and $V$ from \eqref{eq2.8}. By Lemma \ref{lem2.4} and our hypotheses that $b'\geq 0$ and $\alpha\geq 0$,
\begin{align}\label{eq4.1}
V(x)&\leq V(x)I_{(M_1,\infty)}(x)\notag\\
&=V(x)I_{[M_2,\infty)}(x)-\tfrac12\left(b^2(x)+b'(x)\right)I_{(M_1,M_2)}(x)\notag\\
&\leq V(x)I_{[M_2,\infty)}(x).
\end{align}
By Lemma \ref{lem2.3}, since $\beta>0$,
\begin{equation}\label{eq4.2}
h(x)\leq C\exp\left(-\tfrac{\beta}{1-p}x^{1-p}\right), \quad x>0.
\end{equation}
By Lemma \ref{lem2.5}, \eqref{eq4.1}, scaling and translation,
\begin{align}\label{eq4.3}
P_x(\tau_0(X)>t)&=E_x\left[\exp\left(\int_0^tV(B_s)\,ds\right)h(B_t)I_{\tau_0(B)>t}\right]\notag\\
&\leq E_x\left[\exp\left(\int_0^t(VI_{[M_2,\infty)}(B_s)\,ds\right)h(B_t)I_{\tau_0(B)>t}\right]\notag\\
&= E_0\left[\exp\left(t\int_0^1(VI_{[M_2,\infty)}(\sqrt t(B_u+x/\sqrt t))\,du\right)\cdot\right.\notag\\
&\left.\vphantom{\int_0^t}\hspace{.75in}\cdot h(\sqrt t(B_1+x/\sqrt t))I(\tau_0(B+x/\sqrt t)>1)\right].
\end{align}
Writing $\vp=t^{-(1-p)/(1+p)}$ and $x_\vp=x\vp^{1/(1-p)}$, we have
\[
\sqrt t(B_u+x/\sqrt t)=\sqrt tB_u+x=\vp^{-1/(1-p)}(\vp^{1/2}B_u+x_\vp).
\]
By \eqref{eq2.8} and \eqref{eq1.2},
\begin{align*}
VI_{[M_2,\infty)}(ax)&=-\tfrac12\left(\beta^2a^{-2p}x^{-2p}+p\beta a^{-p-1}x^{-p-1}\right)I_{[a^{-1}M_2,\infty)}(x)\\
&\leq-\tfrac12\beta^2a^{-2p}x^{-2p}\,I_{[a^{-1}M_2,\infty)}(x).
\end{align*}
Thus
\begin{align}\label{eq4.4}
(VI_{[M_2,\infty)})&(\sqrt tB_u+x)=(VI_{[M_2,\infty)})(\vp^{-1/(1-p)}(\vp^{1/2}B_u+x_\vp))\notag\\
&\leq-\tfrac12\beta^2\vp^{2p/(1-p)}(\vp^{1/2}B_u+x_\vp)^{-2p}\,I_{\left[\vp^{1/(1-p)}M_2,\infty\right)}(\vp^{1/2}B_u+x_\vp).
\end{align}
By \eqref{eq4.2},
\begin{align}\label{eq4.5}
h(\sqrt tB_1+x)&\leq C\exp\left(-\tfrac{\beta}{1-p}(\sqrt tB_1+x)^{1-p}\right)\notag\\
&=C\exp\left(-\tfrac{\beta}{1-p}\vp^{-1}(\vp^{1/2}B_1+x_\vp)^{1-p}\right).
\end{align}
Substituting $t=\vp^{-(1+p)/(1-p)}$ into \eqref{eq4.3} and using \eqref{eq4.4}--\eqref{eq3.5} gives
\begin{align*}
P_x(\tau_0(X)>t)&\leq CE_0\left[\exp\left(-\tfrac{\beta^2}{2}\vp^{-1}\int_0^1\left(\vp^{1/2}B_u+x_\vp\right)^{-2p}I_{\left[\vp^{1/(1-p)}M_2,\infty\right)}(\vp^{1/2}B_u+x_\vp)\,du\right.\right.\notag\\
&\qquad\left.\left.-\tfrac{\beta}{1-p}\vp^{-1}\left(\vp^{1/2}B_1+x_\vp\right)^{1-p}\right)I(\tau_0(\vp^{1/2} B+x_\vp)>1)\right]\\
&\leq CE_0\left[\exp\left(-\tfrac{\beta^2}{2}\vp^{-1}\int_0^1\left(\vp^{1/2}B_u+x_\vp\right)^{-2p}I_{\left[\vp^{1/(1-p)}M_2,\infty\right)}(\vp^{1/2}B_u+x_\vp)\,du\right.\right.\notag\\
&\qquad\left.\left.-\tfrac{\beta}{1-p}\vp^{-1}\left|\vp^{1/2}B_1+x_\vp\right|^{1-p}\right)\right].
\end{align*}
Writing $Q_\vp$ for the law on $C([0,\infty),\mathbb{R})$ of $\sqrt\vp B$ under $P_0$, this becomes
\begin{align*}
P_x(\tau_0(X)>t)&\leq CE^{Q_\vp}\left[\exp\left(-\tfrac{\beta^2}{2}\vp^{-1}\int_0^1\left(\omega_u+x_\vp\right)^{-2p}I_{\left[\vp^{1/(1-p)}M_2,\infty\right)}(\omega_u+x_\vp)\,du\right.\right.\notag\\
&\qquad\left.\left.-\tfrac{\beta}{1-p}\vp^{-1}\left|\omega_1+x_\vp\right|^{1-p}\vphantom{\int_0^1}\right)\right]\\
&=CE^{Q_\vp}\left[\exp\left(-\tfrac1\vp J_\vp(\omega)\right)\right],
\end{align*}
where
\begin{align*}
J_\vp(\omega)&=\tfrac{\beta^2}{2}\int_0^1\left(\omega_u+x_\vp\right)^{-2p}I_{\left[\vp^{1/(1-p)}M_2,\infty\right)}(\omega_u+x_\vp)\,du+\tfrac{\beta}{1-p}\left|\omega_1+x_\vp\right|^{1-p}.
\end{align*}
Set
\[
J(\omega)=\tfrac{\beta^2}{2}\int_0^1|\omega_u|^{-2p}\,du+\tfrac{\beta}{1-p}|\omega_1|^{1-p}
\]
if the integral is finite, otherwise set $J(\omega)=\infty$. Then $J$ is lower semicontinuous on
\begin{equation}\label{eq4.6}
C_0=\{\omega:[0,1]\to\mathbb{R}\,|\,\omega \text{ is continuous and $\omega(0)=0$}\}
\end{equation}
and if $\omega_n\to\omega$ in $C_0$ as $n\to\infty$, then
\[
\mathop{\liminf_{n\to\infty}}_{\vp\to 0^+}J_\vp(\omega_n)\geq J(\omega).
\]
If $\omega\in C_0$ is absolutely continuous, denote its derivative by $\dot\omega_u$. Then for
\[
K_0=\{\omega\in C_0:\int_0^1|\dot\omega_u|^2\,du<\infty\},
\]
by Varadhan's Theorem (1984, Theorem 2.3),
\begin{align*}
\limsup_{\vp\to 0^+}\vp\log E^{Q_\vp}\left[\exp\left(-\tfrac1\vp J_\vp(\omega)\right)\right]&\leq-\inf_{\omega\in C_0}\left[J(\omega)+\tfrac12\int_0^1(\dot\omega_u)^2\,du\right]\\
&=-\inf_{\omega\in K_0}\left[J(\omega)+\tfrac12\int_0^1|\dot\omega_u|^2\,du\right].
\end{align*}
By Theorem 6.1 in DeBlassie and Smits (2007), the infimum is $\gamma(p,\beta)$. This completes the proof.

\end{proof}


\section{Transformation of Drift}\label{sec5}

In this section we set the stage to extend Theorem \ref{thm4.1} to a drift with nonnegative $\alpha$ and $b$ that can take on positive values on $[M_1,M_2]$. There will be no increasing or $C^3$ conditions imposed. We will also set things up for the case of negative $\alpha$.

\bigskip\n First we prove a variant of the formula in Lemma \ref{lem2.5} that is applicable to discontinuous drifts.

\begin{lem}\label{lem5.1}
Suppose $b_X$ and $b_Y$ satisfy \eqref{eq1.2} with the same $M_1$ and $M_2$. If $b_X$ and $b_Y$ are continuous on $(0,\infty)\backslash\{M_2\}$ and each restricted to $(M_1,M_2)$ has a $C^1$ extension to $[M_1,M_2]$, then the function
\begin{equation}\label{eq5.1}
g(y)=\int_0^y\left(b_X-b_Y\right)(z)\,dz
\end{equation}
is a linear combination of convex functions with generalized second derivative $\mu$ given by
\begin{equation}\label{eq5.2}
\mu(A)=(b_X-b_Y)(M_2-)\,\delta_{M_2}(A)-\int_A\left(\left(b_X'-b_Y'\right)I_{(M_1,M_2)}\right)(a)\,da
\end{equation}
for any Borel set $A\subseteq(0,\infty)$. Here $\delta_{M_2}$ is the unit point mass at $M_2$.
\end{lem} 

\begin{proof}
By our assumptions on $b_X$ and $b_Y$, $b_X-b_Y$ is of bounded variation on any bounded open subset $U$ of $(0,\infty)$, hence $g$ is a linear combination of convex functions on $U$ (Roberts and Varberg (1973), page 23). We now show its generalized second derivative is given by \eqref{eq5.2}.

\bigskip
\n Since $g'=b_X-b_Y$ a.e.,  it suffices to show for each $H\in C_0^\infty(0,\infty)$,
\begin{equation}\label{eq5.3}
\int\left(H'(b_X-b_Y)\right)(a)\,da=\int H(a)\,d\mu(a).
\end{equation}
To this end, let $H$ be so given. Then denoting the $C^1$ extensions of $b_X$, $b_Y$ on $(M_1,M_2)$ to $[M_1,M_2]$ by $\tilde b_X$, $\tilde b_Y$, respectively,
\begin{align*}
\int H'(b_X-b_Y)&=\int H'(\tilde b_X-\tilde b_Y)I_{(M_1,M_2)}
&\intertext{(since $b_X=b_Y$ on $(0,M_1]\cup[M_2,\infty)$)}
&=\lim_{\vp\to0^+}\int H'(\tilde b_X-\tilde b_Y)I_{(M_1+\vp,M_2-\vp)}\\
&=\lim_{\vp\to0^+}\left[(H(\tilde b_X-\tilde b_Y))(M_2-\vp)-(H(\tilde b_X-\tilde b_Y))(M_1+\vp)-\int_{M_1+\vp}^{M_2-\vp}H(\tilde b_X'-\tilde b_Y')\right]\\
&=\lim_{\vp\to0^+}\left[(H(b_X-b_Y))(M_2-\vp)-\left(H(b_X-b_Y)\right)(M_1+\vp)-\int_{M_1+\vp}^{M_2-\vp}H(b_X'-b_Y')\right]\\
&=\left(H(b_X-b_Y)\right)(M_2-)-\int_{M_1}^{M_2}H(b_X'-b_Y')
\intertext{(since $b_X$, $b_Y$ are continuous at $M_1$ and coincide there, and $\tilde b_X'-\tilde b_Y'=b_X'-b_Y'$ is bounded on $(M_1,M_2)$)}
&=H(M_2)(b_X-b_Y)(M_2-)-\int_{M_1}^{M_2}H(b_X'-b_Y')
\end{align*}
(since $H$ is continuous). This gives \eqref{eq5.2}.
\end{proof}

\begin{lem}\label{lem5.2}
Let $x>0$ and let
\begin{equation}\label{eq5.4}
\left\{\begin{array}{ll}
dX_t=dB_t+b_X(X_t)\,dt,&\quad X_0=x\\
\\
dY_t=dB_t+b_Y(Y_t)\,dt,&\quad Y_0=x
\end{array}\right.
\end{equation}
be such that $b_X$ and $b_Y$ satisfy \eqref{eq1.2} with the same $M_1$ and $M_2$. If $b_X$ and $b_Y$ are continuous on $(0,\infty)\backslash\{M_2\}$ and each restricted to $(M_1,M_2)$ has a $C^1$ extension to $[M_1,M_2]$, then for $g$ as in \eqref{eq5.1} and
\[
H=\left(b_X^2-b_Y^2\right)-\left(b_X'-b_Y'\right)I_{(M_1,M_2)}
\]
we have 
\begin{align}\label{eq5.5}
P_x(\tau_0(X)>t)=E_x\left[\exp\left(g(Y_t)-g(x)-\tfrac12\int_0^tH(Y_s)\,ds\right.\right.\notag\\
&\hspace{-2in}\left.\left.-\tfrac12(b_X-b_Y)(M_2-)\,\ell_t^{M_2}(Y)\vphantom{\int_0^t}\right)I(\tau_0(Y)>t)\right],
\end{align}
where $\ell_t^{M_2}(Y)$ is the local time of $Y$ at $M_2$.
\end{lem}

\begin{proof}
The idea is similar to the one used by Pinsky (1995) to get the formula cited in Lemma \ref{lem2.5}: use the Girsanov Theorem and eliminate the stochastic integral in the martingale measure. Our reference for the particular form of the Girsanov Theorem that we use is Theorem 8.6.6 in Oksendal (2007).

\bigskip
\n Let $0<\vp<M_1\leq M_2<M$ such that $x\in(\vp,M)$. Then since $b_X-b_Y$ is bounded on $(\vp,M)$, the ``Local Novikov Condition'' holds:
\[
E_x\left[\exp\left(\tfrac12\int_0^{t\wedge\tau_{\vp,M}(Y)}(b_X-b_Y)^2(Y_s)\,ds\right)\right]<\infty.
\]
It follows that
\[
M_t=\exp\left(\int_0^{t\wedge\tau_{\vp,M}(Y)}(b_X-b_Y)(Y_s)\,dB_s-\tfrac12\int_0^{t\wedge\tau_{\vp,M}(Y)}(b_X-b_Y)^2(Y_s)\,ds\right)
\]
is a martingale. By Girsanov's Theorem,
\begin{align}\label{eq5.6}
P_x(\tau_{\vp,M}(X)>t)&=E_x\left[M_tI(\tau_{\vp,M}(Y)>t)\right]\notag\\
&=E_x\left[\exp\left(\int_0^t(b_X-b_Y)(Y_s)\,dB_s\right.\right.\notag\\
&\qquad\left.\left.-\tfrac12\int_0^t(b_X-b_Y)^2(Y_s)\,ds\right)I(\tau_{\vp,M}(Y)>t)\right].
\end{align}
Now we eliminate the stochastic integral in \eqref{eq5.6}. By Lemma \ref{lem5.1}, we can apply the It\^o-Tanaka formula (Revuz and Yor (1991)) to get for $\tau_{\vp,M}(Y)>t$,
\[
g(Y_t)=g(x)+\int_0^tg_-'(Y_s)\left[dB_s+b_Y(Y_s)\,ds\right]+\tfrac12\int \ell^a_t(Y)\mu(da),
\]
where $g_-'$ is the left derivative and $\mu(da)$ is the generalized second derivative of $g$ from \eqref{eq5.2} in Lemma \ref{lem5.1}. By our hypotheses on $b_X$ and $b_Y$, $g_-'=b_X-b_Y$ a.e. on $(0,\infty)$. Using this, the occupation times formula (Revuz and Yor (1991)) and Lemma \ref{lem5.1} gives, for $\tau_{\vp,M}(Y)>t$,
\begin{align*}
g(Y_t)&=g(x)+\int_0^t(b_X-b_Y)(Y_s)\left[dB_s+b_Y(Y_s)\,ds\right]\\
&\qquad+\tfrac12\left[(b_X-b_Y)(M_2-)\,\ell^{M_2}_t(Y)-\int_0^t((b_X'-b_Y')I_{(M_1,M_2)})(Y_s)\,ds\right]\\
\end{align*}
Solving for the stochastic integral, we get, for $\tau_{\vp,M}(Y)>t$,
\begin{align*}
\int_0^t(b_X-b_Y)(Y_s)\,dB_s&=g(Y_t)-g(x)-\int_0^t((b_X-b_Y)b_Y)(Y_s)\,ds\\
&\qquad+\tfrac12\int_0^t\left((b'_X-b'_Y)I_{(M_1,M_2)}\right)(Y_s)\,ds-\tfrac12(b_X-b_Y)(M_2-)\,\ell_t^{M_2}(Y).
\end{align*}
Substituting this into the exponential in \eqref{eq5.6} gives
\begin{align*}
P_x(\tau_{\vp,M(X)}>t)&=E_x\left[\exp\left(g(Y_t)-g(x)-\int_0^t((b_X-b_Y)b_Y)(Y_s)\,ds\right.\right.\hspace{1in}\notag\\
&\qquad+\tfrac12\int_0^t\left((b'_X-b'_Y)I_{(M_1,M_2)}\right)(Y_s)\,ds-\tfrac12(b_X-b_Y)(M_2-)\,\ell_t^{M_2}(Y)\\
&\qquad\qquad\left.\left.-\tfrac12\int_0^t(b_X-b_Y)^2(Y_s)\,ds\right)I(\tau_{\vp,M}(Y)>t)\right]\\
&=E_x\left[\exp\left(g(Y_t)-g(x)-\tfrac12\int_0^tH(Y_s)\,ds\right.\right.\notag\\
&\qquad\left.\left.-\tfrac12(b_X-b_Y)(M_2-)\,\ell_t^{M_2}(Y)\vphantom{\int_0^t}\right)I(\tau_{\vp,M}(Y)>t)\right].
\end{align*}
By Monotone Convergence and Lemma \ref{lem2.1}, upon letting $\vp\downarrow0$ and $M\uparrow\infty$, we get the desired conclusion.
\end{proof}

\bigskip
\n Next, we prove a variant of the previous results applicable to the case $\alpha<0$.

\begin{lem}\label{lem5.3}
Let $b_1$ and $b_2$ satisfy \eqref{eq1.2} with the same $\beta>0$, $M_1$ and $M_2$, but the corresponding $\alpha$'s---call them $\alpha_1$ and $\alpha_2$, respectively---are different. Suppose $b_1$ and $b_2$ are continuous on $(0,\infty)$ and for some $M\in(M_1,M_2)$, $b_1-b_2$ restricted to $(M_1,M)$ and $(M,M_2)$ has $C^1$ extensions to $[M_1,M]$ and $[M,M_2]$, respectively. Then the function
\begin{equation}\label{eq5.7}
g(y)=\int_0^y\left(b_1-b_2\right)(z)\,dz
\end{equation}
is a linear combination of convex functions with generalized second derivative $\mu$ given by
\begin{equation}\label{eq5.8}
\mu(A)=-\int_A\left(\left(b_1'-b_2'\right)I_{(M_1,M_2)}\right)(a)\,da
\end{equation}
for any Borel set $A\subseteq(0,\infty)$. 
\end{lem} 

\begin{proof}
This is similar to the proof of Lemma \ref{lem5.1}, except because of the continuity of $b_1$ and $b_2$ on $(0,\infty)$, there will be no point mass term in the generalized second derivative.
\end{proof}

\begin{lem}\label{lem5.4}
Let $x>0$ and let
\begin{equation}\label{eq5.9}
\left\{\begin{array}{ll}
dX_t=dB_t+b_X(X_t)\,dt,&\quad X_0=x\\
\\
dY_t=dB_t+b_Y(Y_t)\,dt,&\quad Y_0=x
\end{array}\right.
\end{equation}
be such that $b_1$ and $b_2$ satisfy \eqref{eq1.2} with the same $\beta>0$, $M_1$ and $M_2$, but the corresponding $\alpha$'s---call them $\alpha_1$ and $\alpha_2$, respectively---are different. Suppose $b_1$ and $b_2$ are continuous on $(0,\infty)$ and for some $M\in(M_1,M_2)$, $b_1-b_2$ restricted to $(M_1,M)$ and $(M,M_2)$ has $C^1$ extensions to $[M_1,M]$ and $[M,M_2]$. Then for $g$ as in \eqref{eq5.7} and
\[
H=\left(b_1^2-b_2^2\right)-\left(b_1'-b_2'\right)I_{(M_1,M_2)\backslash\{M\}}
\]
we have 
\begin{equation}\label{eq5.10}
P_x(\tau_0(X)>t)=E_x\left[\exp\left(g(Y_t)-g(x)-\tfrac12\int_0^tH(Y_s)\,ds\right)I(\tau_0(Y)>t)\right].
\end{equation}
\end{lem}
\begin{proof}
This is a simple modification of the proof of Lemma \ref{lem5.1}, where again, because of the continuity of $b_1$ and $b_2$ on $(0,\infty)$, there will be no local time term. 
\end{proof}

\section{Upper Bound}\label{sec6}

The main result of this section is the following theorem, which combined with Theorem \ref{thm3.1} will prove Theorem \ref{thm1.1}.

\begin{thm}\label{thm6.1} Let $X_t$ be as in \eqref{eq1.1}, where $b$ is from \eqref{eq1.2}. Then
\[
\limsup_{t\to\infty}t^{-(1-p)/(1+p)}\log P_x(\tau_0(X)>t)\leq-\,\gamma(p,\beta),
\]
where $\gamma(p,\beta)$ is from \eqref{eq1.4}.

\end{thm}

\bigskip
\n We break up the proof into the cases when $\alpha>0$, $\alpha<0$ and $\alpha=0$. 

\bigskip\n\emph{First case: $\alpha>0$.} 

\n We remove the assumptions that $b\in C^3$ and $b'\geq 0$ in Theorem \ref{thm4.1}.

\begin{lem}\label{lem6.2} Under the condition \eqref{eq1.2}, if $\alpha>0$, then the solution $X_t$ of \eqref{eq1.1} satisfies
\[
\limsup_{t\to\infty}\log P_x(\tau_0(X)>t)\leq-\,\gamma(p,\beta),
\]
where $\gamma(p,\beta)$ is from \eqref{eq1.4}. 
\end{lem}

\n The proof uses several lemmas. Here is the basic setup. Let $0<x_1<x_2$, $y_1<y_2<0$ and $\gamma>0$ be given. Consider the lines
\begin{equation}\label{eq6.1}
\tilde b_X(x)=\frac{\gamma-y_1}{x_2-x_1}(x-x_1)+y_1,\quad x>0
\end{equation}
and
\begin{equation}\label{eq6.2}
\tilde b_Y(x)=\frac{y_2-y_1}{x_2-x_1}(x-x_1)+y_1,\quad x>0.
\end{equation}
Then $\tilde b_X$ passes through the points $(x_1,y_1)$ and $(x_2,\gamma)$ and $\tilde b_Y$ passes through $(x_1,y_1)$ and $(x_2,y_2)$.
\begin{lem}\label{lem6.3}
For $\gamma$ sufficiently large,
\[
\tilde b_X^2-\tilde b_Y^2-(\tilde b_X'-\tilde b_Y')\geq 0\text{ on $[x_1,x_2]$.}
\]
\end{lem}
\begin{proof}
We have
\[
\tilde b_X^2-\tilde b_Y^2=\left(\frac{x-x_1}{x_2-x_1}\right)^2\left[(\gamma-y_1)^2-(y_2-y_1)^2\right]+\frac{2y_1(\gamma-y_2)}{x_2-x_1}(x-x_1)
\]
and $\tilde b_X^2-\tilde b_Y^2$ takes on its minimum value $-\frac{y_1^2(\gamma-y_2)^2}{(\gamma-y_1)^2-(y_2-y_1)^2}$ at the point $x=x_1-\frac{y_1(\gamma-y_2)(x_2-x_1)}{(\gamma-y_1)^2-(y_2-y_1)^2}$ (note that $\gamma-y_1>y_2-y_1>0$). In particular, on $[x_1,x_2]$,
\begin{align*}
(\tilde b_X^2-\tilde b_Y^2)(x)&+(\tilde b_X'-\tilde b_Y')(x)\geq -\frac{y_1^2(\gamma-y_2)^2}{(\gamma-y_1)^2-(y_2-y_1)^2}+\frac{\gamma-y_1}{x_2-x_1}-\frac{y_2-y_1}{x_2-x_1}\\
&=-\frac{y_1^2(\gamma-y_2)^2}{(\gamma-y_1)^2-(y_2-y_1)^2}+\frac{\gamma-y_2}{x_2-x_1}\\
&=(\gamma-y_2)\left[-\frac{y_1^2(\gamma-y_2)}{(\gamma-y_1)^2-(y_2-y_1)^2}+\frac{1}{x_2-x_1}\right]\\
&\geq0
\end{align*}
for large $\gamma$.
\end{proof}

\begin{lem}\label{lem6.4}
Let $dZ_t=dB_t+b_Z(Z_t)\,dt$, where $b_Z$ satisfies \eqref{eq1.2} with $b_Z=C>0$ on $(M_1,M_2)$. Then
\[
\limsup_{t\to\infty}\log P_x(\tau_0(Z)>t)\leq-\,\gamma(p,\beta),
\]
where $\gamma(p,\beta)$ is from \eqref{eq1.4}.
\end{lem}
\begin{proof}
By our hypotheses,
\begin{equation}\label{eq6.3}
b_Z(x)=\left\{\begin{array}{ll}
-\alpha x^{-q},&\quad x\in(0,M_1]\\
C,&\quad x\in(M_1,M_2)\\
-\beta x^{-p},&\quad x\in[M_2,\infty).
\end{array}
\right.
\end{equation}
In Lemma \ref{lem5.2}, choose $(x_2,y_2)=(M_2,-\beta M_2^{-p})$ and $(x_1,y_1)=\left(\tfrac{M_1}{K},-\alpha\left(\tfrac{M_1}{K}\right)^{-q}\right)$ where $K>1$ is so large that $y_1<y_2$. Notice this tells us $y_1<-\alpha M_1^{-q}$. Then we can choose $\gamma>0$ so large that the conclusion of Lemma \ref{lem5.2} holds for $\tilde b_X$ and $\tilde b_Y$ given by \eqref{eq6.1} and \eqref{eq6.2}:                          
\begin{equation}\label{eq6.4}
\tilde b_X^2-\tilde b_Y^2-(\tilde b_X'-\tilde b_Y')\geq 0\text{ on $[x_1,x_2]$.}
\end{equation}
Define
\begin{equation}\label{eq6.5}
b_X(x)=\left\{\begin{array}{ll}
-\alpha x^{-q},&\quad x\in(0,x_1]\\
\tilde b_X(x),&\quad x\in(x_1,x_2)=(x_1,M_2)\\
-\beta x^{-p},&\quad x\in[x_2,\infty)=[M_2,\infty)
\end{array}
\right.
\end{equation}
and
\begin{equation}\label{eq6.6}
b_Y(x)=\left\{\begin{array}{ll}
-\alpha x^{-q},&\quad x\in(0,x_1]\\
\tilde b_Y(x),&\quad x\in(x_1,x_2)=(x_1,M_2)\\
-\beta x^{-p},&\quad x\in[x_2,\infty)=[M_2,\infty).
\end{array}
\right.
\end{equation}
Then by definition of $b_X$, $b_Y$ and \eqref{eq6.4},
\begin{equation}\label{eq6.7}
b_X^2-b_Y^2-(b_X'-b_Y')I_{(x_1,x_2)}\geq 0\text{ on $(0,\infty)$.}
\end{equation}
Notice this continues to hold if $\gamma$ is made larger. Since $x_1=\frac{M_1}{K}<M_1$, we can make $\gamma$ larger, if necessary, so that the line segment $\{(x,C):M_1<z<M_2\}$ (recall $C$ is from \eqref{eq6.3}) lies to the right of the line through the points $(x_1,y_1)$ and $(x_2,\gamma)$. Since $x_1<M_1$, we have $b_Z\leq b_X$ on $(0,\infty)$. By the Comparison Theorem,
\begin{equation}\label{eq6.8}
P_x(\tau_0(Z)>t)\leq P_x(\tau_0(X)>t).
\end{equation}
We now apply Lemma \ref{lem5.2} to $b_X$, $b_Y$ given by \eqref{eq6.5}, \eqref{eq6.6}, respectively, but with $M_1,M_2$ in the lemma taken to be $x_1,x_2$. First check the hypotheses of the Lemma:
\begin{itemize}
\item $b_X$ and $b_Y$ are continuous on $(0,\infty)\backslash\{x_2\}$ since
\[
\lim_{x\to x_1^-}b_X(x)=y_1=\lim_{x\to x_1^+}b_X(x)\text{ and }\lim_{x\to x_1^-}b_Y(x)=y_1=\lim_{x\to x_1^+}b_Y(x).
\]
\item Since each is linear on $(x_1,x_2)$, each restricted to $(x_1,x_2)$ has a $C^1$ extension to $[x_1,x_2]$.
\end{itemize}
Thus Lemma \ref{lem5.2} applies and so for
\[
g(z)=\int_0^x(b_X-b_Y)(y)\,dy
\]
and
\[
H=(b_X^2-b_Y^2)-(b_X'-b_Y')I_{(x_1,x_2)},
\]
we have 
\begin{align}\label{eq6.9}
P_x(\tau_0(X)>t)=E_x\left[\exp\left(g(Y_t)-g(x)-\tfrac12\int_0^tH(Y_s)\,ds\right.\right.\notag\\
&\hspace{-2in}\left.\left.-\tfrac12\left(b_X-b_Y\right)(M_2-)\,\ell_t^{M_2}(Y)\vphantom{\int_0^t}\right)I(\tau_0(Y)>t)\right].
\end{align}
Now by \eqref{eq6.7}, $H\geq0$, and by definition of $b_X$ and $b_Y$,
\begin{align*}
(b_X-b_Y)(M_2-)&=\lim_{x\to x_2^-}\left(\tilde b_X(x)-\tilde b_Y(x)\right)\text{ (by \eqref{eq6.5}--\eqref{eq6.6})}\\
&=\tilde b_X(x_2)-\tilde b_Y(x_2)\text{ ($\tilde b_X$ and $\tilde b_Y$ are continuous)}\\
&=\gamma-y_2\text{ (by \eqref{eq6.1}--\eqref{eq6.2})}\\
&>0.
\end{align*}
Then \eqref{eq6.9} becomes
\begin{equation}\label{eq6.10}
P_x(\tau_0(X)>t)\leq E_x\left[\exp\left(g(Y_t)-g(x)\right)I(\tau_0(Y)>t)\right].
\end{equation}
Since $b_X=b_Y$ on $(x_1,x_2)^c$ and since $b_X-b_Y=\tilde b_X-\tilde b_Y$ is bounded on $(x_1,x_2)$, we see
\[
g(z)=\int_0^z(b_X-b_Y)(y)\,dy\leq \sup_{(x_1,x_2)}(\tilde b_X-\tilde b_Y)(x_2-x_1),\quad z>0.
\]
Hence for some positive constant $C_1$ independent of $t$, \eqref{eq6.10} becomes
\begin{equation}\label{eq6.11}
P_x(\tau_0(X)>t)\leq C_1P_x(\tau_0(Y)>t).
\end{equation}
The crucial point is that
\[
\sup_{(0,\infty)}b_Y<0.
\]
This holds because:
\begin{itemize}
\item on $(0,x_1]$, $b_Y(x)=-\alpha x^{-q}\leq-\alpha x_1^{-p}$;
\item on $(x_1,x_2)$, by \eqref{eq6.2} and that $y_1<y_2<0$, $b_Y=\tilde b_Y\leq y_2<0$;
\item on $(x_2,\infty)$, $b_Y(x)=-\beta x^{-p}\leq -\beta x_2^{-p}$.
\end{itemize}
Thus Theorem \ref{thm4.1} applies to the process $Y$, and we conclude
\[
\limsup_{t\to\infty}\log P_x(\tau_0(Y)>t)\leq-\,\gamma(p,\beta),
\]
where $\gamma(p,\beta)$ is from \eqref{eq1.4}. Combined with \eqref{eq6.8} and \eqref{eq6.11}, this gives us
\[
\limsup_{t\to\infty}\log P_x(\tau_0(Z)>t)\leq-\,\gamma(p,\beta),
\]
as desired.
\end{proof}

\n{\bf Proof of Lemma \ref{lem6.2}.} Let $X_t$ be from \eqref{eq1.1}, where $b$ satisfies \eqref{eq1.2} and $\alpha\geq 0$. Let $C=\sup_{[M_1,M_2]}|b|$. Suppose $dZ_t=dB_t+b_Z(Z_t)\,dt$, where $b_Z$ satisfies \eqref{eq1.2} with $b_Z=C$ on $(M_1,M_2)$. Then $b_X\leq b_Z$ and so by the Comparison Theorem, $P_x(\tau_0(X)>t)\leq P_x(\tau_0(Z)>t)$. The desired upper bound follows upon applying Lemma \ref{lem6.4} to $Z$.\hfill$\square$

\bigskip\n\emph{Second case: $\alpha<0$.}
\begin{lem}\label{lem6.5} Under the condition \eqref{eq1.2}, if $\alpha<0$, then the solution $X_t$ of \eqref{eq1.1} satisfies
\[
\limsup_{t\to\infty}\log P_x(\tau_0(X)>t)\leq-\,\gamma(p,\beta),
\]
where $\gamma(p,\beta)$ is from \eqref{eq1.4}. 
\end{lem}
\n For the proof, we need the following special case.
\begin{lem}\label{lem6.6} Let $\alpha<0$. Assume $b\in C^3$ satisfies condition \eqref{eq1.2}. Suppose that for some $x_1\in (0,M_2)$, $b(x_1)=0$ and $b$ is strictly decreasing on $(0,x_1)$. Then the solution $X_t$ of \eqref{eq1.1} satisfies
\[
\limsup_{t\to\infty}\log P_x(\tau_0(X)>t)\leq-\,\gamma(p,\beta),
\]
where $\gamma(p,\beta)$ is from \eqref{eq1.4}. 
\end{lem}

\begin{proof}
Define
\begin{equation}\label{eq6.12}
b_Y(x)=\left\{\begin{array}{ll}
-b(x),&\quad x\in(0,x_1]\\
b(x),&\quad x\in(x_1,\infty),
\end{array}
\right.
\end{equation}
and let $Y_t$ solve
\[
dY_t=dB_t+b_Y(Y_t)\,dt,\quad Y_0=x.
\]
Now $b$ and $b_Y$ are continuous on $(0,\infty)$ and $b-b_Y$ restricted to $(M_1,x_1)$ and $(x_1,M_2)$ has $C^1$ extensions to $[M_1,x_1]$ and $[x_1,M_2]$, respectively. Then by Lemma \ref{lem5.4}, taking $M$ there to be $x_1$, for
\[
g(z)=\int_0^z\left(b-b_Y\right)(y)\,dy
\]
and
\[
H=\left(b^2-b_Y^2\right)-\left(b'-b_Y'\right)I_{(M_1,M_2)\backslash\{x_1\}},
\]
we have 
\begin{equation}\label{eq6.13}
P_x(\tau_0(X)>t)=E_x\left[\exp\left(g(Y_t)-g(x)-\tfrac12\int_0^tH(Y_s)\,ds\right)I(\tau_0(Y)>t)\right].
\end{equation}
On the other hand, $b^2=b_Y^2$ on $(0,\infty)$ and
\begin{align*}
(b'-b_Y')(x)&=\left\{\begin{array}{ll}
2b'(x),&\quad x\in(M_1,x_1)\\
0,&\quad x\in(x_1,\infty)
\end{array}
\right.\\
&\leq 0.
\end{align*}
In particular, $H\geq 0$ and \eqref{eq6.13} becomes
\begin{equation}\label{eq6.14}
P_x(\tau_0(X)>t)\leq E_x\left[\exp\left(g(Y_t)-g(x)\right)I(\tau_0(Y)>t)\right].
\end{equation}
Since $b=b_Y$ on $(x_1,\infty)$ and $b-b_Y=2b$ is nonnegative and integrable on $(0,x_1)$, we see $g$ is bounded. Hence for some positive constant $C_1$ independent of $t$, \eqref{eq6.14} becomes
\begin{equation}\label{eq6.15}
P_x(\tau_0(X)>t)\leq C_1P_x(\tau_0(Y)>t).
\end{equation}
By Lemma \ref{lem6.2} applied to $Y$,
\[
\limsup_{t\to\infty}\log P_x(\tau_0(X)>t)\leq\limsup_{t\to\infty}\log P_x(\tau_0(Y)>t)\leq-\,\gamma(p,\beta),
\]
where $\gamma(p,\beta)$ is from \eqref{eq1.4}. 
\end{proof}

\bigskip
\n{\bf Proof of Lemma \ref{lem6.5}.} Let $X_t$ be from \eqref{eq1.1}, where $b$ satisfies \eqref{eq1.2} and $\alpha<0$. Let $C>\sup_{[M_1,M_2]}|b|$ and choose $\alpha_1<\alpha$ such that $-\alpha_1>CM_2^q$. Then choose $b_Y\in C^3$ having the following properties: for some $M_3,M_4\in(M_2,\infty)$ with $M_3<M_4$,
\begin{itemize}
\item $b_Y(x)=-\alpha_1x^{-q}$ on $(0,M_2]$;
\item $b_Y$ is strictly decreasing on $[M_2,M_3]$;
\item $b_Y(M_3)=0$;
\item $b_Y=b$ on $[M_4,\infty)$.
\end{itemize}

\n Suppose $dY_t=dB_t+b_Y(Y_t)\,dt$. Then $b\leq b_Y$ and so by the Comparison Theorem, $P_x(\tau_0(X)>t)\leq P_x(\tau_0(Y)>t)$. The process $Y$ satisfies the conditions of Lemma \ref{lem6.6} and the desired upper bound follows by applying that lemma to $Y$.\hfill$\square$

\bigskip\n\emph{Third case: $\alpha=0$.}

\n Define
\[
b_Y(x)=\left\{\begin{array}{ll}
x^{-q},&\quad x\in(0,M_1]\\
b(x),&\quad x\in(M_1,\infty),
\end{array}
\right.
\]
and let $Y_t$ solve
\[
dY_t=dB_t+b_Y(Y_t)\,dt,\quad Y_0=x.
\]
Then by the Comparison Theorem,
\[
P_x(\tau_0(X)>t)\leq P_x(\tau_0(Y)>t).
\]
Since the situation of $\alpha<0$ holds for $b_Y$, that case implies
\[
\limsup_{t\to\infty}\log P_x(\tau_0(X)>t)\leq\limsup_{t\to\infty}\log P_x(\tau_0(Y)>t)\leq-\,\gamma(p,\beta),
\]
once again. \hfill$\square$

\section{Proof of Theorem \ref{thm1.2}}\label{sec7}

\n We need the following bounds due to Potter. See Theorem 1.5.6 in Bingham, Goldie and Teugels (1989).
\begin{thm}\label{thm7.1}
(a) If $\ell$ is slowly varying at infinity, then for any choice of $A>1$ and $\delta>0$, there is $M>0$ such that
\[
\ell(y)/\ell(x)\leq A\max\left((y/x)^\delta,(y/x)^{-\delta}\right),\quad x,y\geq M.
\]
(b) If $\ell$ is slowly varying at zero from the right, then for any choice of $A>1$ and $\delta>0$, there is $M>0$ such that
\[
\ell(y)/\ell(x)\leq A\max\left((y/x)^\delta,(y/x)^{-\delta}\right),\quad x,y\leq M.
\]
\end{thm}

\bigskip\n The following is a consequence of the Representation Theorem for slowly varying functions (Bingham, Goldie and Teugels (1989), Theorem 1.3.1).
\begin{thm}\label{thm7.2}
(a) If $\ell$ is slowly varying at infinity, then for each $\delta>0$
\[
\lim_{x\to\infty}x^{-\delta}\ell(x)=0\text{ and } \lim_{x\to\infty}x^{\delta}\ell(x)=\infty.
\]
(b) If $\ell$ is slowly varying from the right at $0$, then for each $\delta>0$
\[
\lim_{x\to 0^+}x^{\delta}\ell(x)=0 \text{ and } \lim_{x\to 0^+}x^{-\delta}\ell(x)=\infty.
\]
\end{thm}

\bigskip
\n We will also need the following fact to prove Theorem \ref{thm1.2}.

\begin{lem}\label{lem7.3}
Let $\delta>0$ be so small that $0<p-\delta<p+\delta<1$ and $0<q-\delta<q+\delta<1$. By decreasing $M_1$ and increasing $M_2$ if necessary, we have
\[
x^{-\delta}\leq\ell_2(x)\leq x^\delta,\quad x\geq M_2
\]
and
\[
x^{\delta}\leq\ell_1(x)\leq x^{-\delta},\quad x\leq M_1.
\]
\end{lem}

\begin{proof} By making $M_2$ bigger if necessary, by Theorem \ref{thm7.2}(a),
\begin{equation}\label{eq7/1}
M_2^{-\delta}\ell_2(M_2)\leq 1/2<2\leq M_2^{\delta}\ell_2(M_2).
\end{equation}
Then by Theorem \ref{thm7.1}(a), again making $M_2$ bigger if necessary, we have for $x\geq M_2$,
\begin{align*}
\ell_2(x)&\leq 2\max\left((x/M_2)^{\delta},(x/M_2)^{-\delta}\right)\ell_2(M_2)\\
&=2(x/M_2)^\delta\ell_2(M_2)\\
&=2x^\delta M_2^{-\delta}\ell_2(M_2)\\
&\leq 2x^\delta(1/2)\\
&= x^\delta
\end{align*}
and
\begin{align*}
\ell_2(M_2)&\leq 2\max\left((M_2/x)^{\delta},(M_2/x)^{-\delta}\right)\ell_2(x)\\
&=2(x/M_2)^{\delta}\ell_2(x)\\
&\leq (M_2^\delta\ell_2(M_2))(x/M_2)^{\delta}\ell_2(x)\\
&= x^\delta\ell_2(M_2)\ell_2(x).
\end{align*}
Rearranging the latter, we have
\[
x^{-\delta}\leq \ell_2(x).
\]
Combined, we have
\begin{equation}\label{eq7.2}
x\geq M_2\implies x^{-\delta}\leq\ell_2(x)\leq x^\delta.
\end{equation} 

\bigskip\n
Similarly, by making $M_1$ smaller if necessary, by Theorems \ref{thm7.2}(b) and \ref{thm7.1}(b) we have
\begin{equation}\label{eq7.3}
x\leq M_1\implies x^{\delta}\leq\ell_1(x)\leq x^{-\delta}.
\end{equation}
\end{proof}

\bigskip\n {\emph{Proof of Theorem 1.2.} Let $\vp>0$ be small so that $0<\beta_0-\vp<\beta_0+\vp<1$, and if necessary, decrease $M_1$ and increase $M_2$ so that Lemma \ref{lem7.3} continues to apply and we have 
\[
\alpha_0=\sup_{(0,M_1]}|\alpha|<\infty
\]
and 
\[
x\geq M_2\implies 0<\beta_0-\vp<\beta(x)<\beta_0+\vp<1.
\]
Define
\[
\widetilde b(x)=\left\{
\begin{array}{ll}
\alpha_0 x^{q-\delta},&\qquad 0<x\leq M _1\\
b(x),&\qquad M_1<x<M_2\\
-(\beta_0+\vp) x^{p+\delta},&\qquad M_2\leq x,
\end{array}
\right.
\]
and 
\[
\overline b(x)=\left\{
\begin{array}{ll}
-\alpha_0 x^{q-\delta},&\qquad 0<x\leq M _1\\
b(x),&\qquad M_1<x<M_2\\
-(\beta_0-\vp) x^{p-\delta},&\qquad M_2\leq x.
\end{array}
\right.
\]
Let $\widetilde X_t$ be the solution of
\[
d\widetilde X_t=dB_t+\widetilde b(\widetilde X_t)\,dt,\quad \widetilde X_0=x
\]
and let $\overline X_t$ be the solution of
\[
d\overline X_t=dB_t+\overline b(\overline X_t)\,dt,\quad \overline X_0=x.
\]
Since $\overline b\leq b\leq\widetilde b$, by the Comparison Theorem we have
\[
P_x(\tau_0(\overline X)>t)\leq P_x(\tau_0(X)>t)\leq P_x(\tau_0(\widetilde X)>t).
\]
But Theorem \ref{thm1.1} applies to $\widetilde X_t$ and $\overline X_t$ and so we get
\begin{align*}
-\gamma(p-\delta,\beta_0-\vp)&\leq\liminf_{t\to\infty}t^{-(1-p)/(1+p)}\log P_x(\tau_0(X)>t)\\
&\leq\limsup_{t\to\infty}t^{-(1-p)/(1+p)}\log P_x(\tau_0(X)>t)\\
&-\gamma(p+\delta,\beta_0+\vp).
\end{align*}
Let $\vp$ and $\delta$ go to $0$ to get the desired conclusion.\hfill$\square$



\end{document}